\documentclass[11pt]{amsart}
\usepackage{amsmath,amssymb}

\usepackage{fullpage}
\usepackage{amsfonts,amsthm,url,tabularx,array,mathtools,enumerate,verbatim,float} 
\usepackage{tikz} 

\usepackage{thmtools}
\usetikzlibrary{arrows} 
\newcommand\numberthis{\addtocounter{equation}{1}\tag{\theequation}} 
\def\COMMENT#1{}
\def\TASK#1{}
\def\noproof{{\unskip\nobreak\hfill\penalty50\hskip2em\hbox{}\nobreak\hfill%
        $\square$\parfillskip=0pt\finalhyphendemerits=0\par}\goodbreak}
\def\endproof{\noproof\bigskip}
\newdimen\margin   
\def\textno#1&#2\par{%
    \margin=\hsize
    \advance\margin by -4\parindent
           \setbox1=\hbox{\sl#1}%
    \ifdim\wd1 < \margin
       $$\box1\eqno#2$$%
    \else
       \bigbreak
       \hbox to \hsize{\indent$\vcenter{\advance\hsize by -3\parindent
       \sl\noindent#1}\hfil#2$}%
       \bigbreak
    \fi}
\def\proof{\removelastskip\penalty55\medskip\noindent{\bf Proof. }}

\def\eps{\varepsilon}

\def\LL{\mathcal{L}}

\DeclarePairedDelimiter\floor{\lfloor}{\rfloor} 
\DeclarePairedDelimiter\ceil{\lceil}{\rceil}

\newtheorem{thm}{Theorem}
\newtheorem{fact}[thm]{Fact}
\newtheorem{lemma}[thm]{Lemma}
\newtheorem{prop}[thm]{Proposition}
\newtheorem{cor}[thm]{Corollary}

\newtheorem{claim}[thm]{Claim}

\declaretheoremstyle[notefont=\bfseries,notebraces={}{},%
    headpunct={},postheadspace=1em]{mystyle}
\declaretheorem[style=mystyle,numbered=no,name=Theorem]{thm-hand}

\allowdisplaybreaks

\title{On solution-free sets of integers}
\author{Robert Hancock and  Andrew Treglown}
\begin{document}
\label{firstpage}
\begin{abstract} 
Given a linear equation $\LL$, a set $A \subseteq [n]$
is $\LL$-free if $A$ does not contain any `non-trivial' solutions to $\LL$. In this paper we consider  the following three general questions:
\begin{itemize}
\item[(i)] What is the size of the largest $\LL$-free subset of $[n]$?
\item[(ii)] How many $\LL$-free subsets of $[n]$ are there?
\item[(iii)] How many maximal $\LL$-free subsets of $[n]$ are there?
\end{itemize}
We completely resolve (i) in the case when $\LL$ is the equation $px+qy=z$ for fixed $p,q\in \mathbb N$ where $p\geq 2$. Further, up to a multiplicative constant, we answer (ii) for a wide class of such equations $\LL$, thereby refining a special case of a result of
Green~\cite{G-R}.
We also give various bounds on the number of maximal $\LL$-free subsets of $[n]$ for three-variable homogeneous linear equations $\LL$. For this, we make use of container and removal lemmas of Green~\cite{G-R}.
\end{abstract}
\date{\today}
\maketitle

\section{Introduction} 
Let $[n]:=\{1, \dots, n\}$ and consider a fixed linear equation
$\LL$ of the form 
\begin{align}\label{Leq}
a_1x_1+\dots +a_k x_k =b
\end{align}
 where $a_1, \dots ,a_k ,b \in \mathbb Z$. 
If $b=0$ we say that $\LL$ is \emph{homogeneous}.
If 
$$\sum _{i \in [k]} a_i=b=0$$ then we say that $\LL$ is \emph{translation-invariant}.
Let $\LL$ be translation-invariant. Then notice that $(x,\dots, x)$ is a `trivial' solution of (\ref{Leq}) for any $x$.
More generally, a solution $(x_1, \dots , x_k)$ to $\LL$ is said to be \emph{trivial} if there exists a partition $P_1, \dots ,P_\ell$ of $[k]$ so that: 
\begin{itemize}
\item[(i)] $x_i=x_j$ for every $i,j$ in the same partition class $P_r$; 
\item[(ii)] For each $r \in [\ell]$, $\sum _{i \in P_r} a_i=0$. 
\end{itemize}
A set $A \subseteq [n]$ is \emph{$\LL$-free} if $A$ does not contain any non-trivial solutions to $\LL$. If the equation $\LL$ is clear from the context, then we simply say $A$ is \emph{solution-free}.

The notion of an $\LL$-free set encapsulates many  fundamental topics in combinatorial number theory. Indeed,
in the case when $\LL$ is $x_1+x_2=x_3$ we call an $\LL$-free set a \emph{sum-free set}. This is a notion that dates back to 1916 when Schur~\cite{schur} proved that, if $n$ is sufficiently large, any $r$-colouring of  $[n]$ yields a monochromatic triple $x,y,z$ such that $x+y=z$. \emph{Sidon sets} (when $\LL$ is $x_1+x_2=x_3+x_4$) have also been extensively studied. For example, a classical result of Erd\H{o}s and Tur\'an~\cite{erdos} asserts that the largest Sidon set in $[n]$ has size $(1+o(1))\sqrt{n}$.
In the case when $\LL$ is $x_1+x_2=2x_3$ an $\LL$-free set is simply a \emph{progression-free} set. Roth's theorem~\cite{roth} states that the largest progression-free subset of $[n]$ has size $o(n)$.
In \cite{ruzsa, ruzsa2}, Ruzsa instigated the study of solution-free sets for general linear equations.

In this paper we prove a number of results concerning $\LL$-free subsets of $[n]$ where $\LL$ is a homogeneous linear equation in \emph{three variables}. In particular, our work is motivated by the following general questions:
\begin{itemize}
\item[(i)] What is the size of the largest $\LL$-free subset of $[n]$?
\item[(ii)] How many $\LL$-free subsets of $[n]$ are there?
\item[(iii)] How many \emph{maximal} $\LL$-free subsets of $[n]$ are there?
\end{itemize}
We make progress on all three of these questions. For each question we 
use tools from graph theory; for (i) and (ii) our methods are somewhat elementary. For (iii) our method is more involved and utilises container and removal lemmas of Green~\cite{G-R}.

\subsection{The size of the largest solution-free set}
As highlighted above, a central question in the study of $\LL$-free sets is to establish the size $\mu _{\LL} (n)$ of the largest $\LL$-free subset of $[n]$. It is not difficult to see that the largest sum-free  subset of $[n]$ has size $\lceil n/2 \rceil$, and this bound is attained by the set of odd numbers in $[n]$ and by the interval $[\lfloor n/2 \rfloor+1,n]$. 

When $\LL$ is $x_1+x_2=2x_3$, $\mu _{\LL} (n)=o(n)$ by Roth's theorem. 
In fact, very recently Bloom~\cite{bloom} proved that there is a constant $C$ such that every set $A\subseteq [n]$ with $|A|\geq C n(\log \log n)^4 / \log n$ contains a three-term arithmetic progression. On the other hand, Behrend~\cite{beh} showed that there is a constant $c>0$ so that $\mu _{\LL} (n)\geq n \exp(-c\sqrt{\log n})$. See~\cite{elk, greenwolf} for the best known lower bound on $\mu _{\LL} (n)$ in this case.

More generally, it is known that $\mu _{\LL} (n)=o(n)$ if $\LL$ is translation-invariant and  $\mu _{\LL} (n)=\Omega(n)$ otherwise (see~\cite{ruzsa}). For other (exact) bounds on $\mu _{\LL }(n)$ for various linear equations $\LL$ see, for example,~\cite{ruzsa, ruzsa2, baltz, dil, hegarty}. 

In this paper we mainly focus on $\LL$-free subsets of $[n]$ for linear equations $\LL$ of the form $px+qy=z$ where $p\geq 2$ and $q \geq 1$ are fixed integers. 
Notice that for such a linear equation $\LL$, the interval $[\lfloor n/(p+q)\rfloor +1, n]$ is an $\LL$-free set. 
Our first result implies that this is the largest such $\LL$-free subset of $[n]$.
Let $\min(S)$ denote the smallest element in a finite set $S\subseteq \mathbb N$.

\begin{thm}\label{structure}
Let $\LL$ denote the equation $px+qy=z$ where $p\geq q$ and $p\geq 2$, $p,q \in \mathbb{N}$. Let $S$ be an $\LL$-free subset of $[n]$, and let $\min(S)=\floor{\frac{n}{p+q}}-t$ where $t$ is a non-negative integer. 

\begin{enumerate}[(i)]
\item{If $0 \leq t < (\frac{p+q-1}{p+q+p/q})\floor{\frac{n}{p+q}}$ then $|S| \leq \ceil{\frac{(p+q-1)n}{p+q}}-\floor{\frac{p}{q}t}$.}
\item{If $t \geq (\frac{p+q-1}{p+q+p/q})\floor{\frac{n}{p+q}}$ then $|S| \leq \frac{(q^2+1)n}{q^2+q+1}$ provided that $$n \geq \max{\Big\{\frac{3(q^2+q+1)(q^3+p(q^2+q+1))}{q^2+1}, \frac{5(q^2+q+1)(q^5+p(q^4+q^3+q^2+q+1))}{q^4+(p-1)q^3+q^2+1}\Big\}}.$$}
\end{enumerate}
\end{thm}

In both cases of Theorem~\ref{structure} we observe that $|S| \leq n-\lfloor \frac{n}{p+q} \rfloor$, hence the following corollary holds.

\begin{cor}\label{cormain}
Let $\LL$ denote the equation $px+qy=z$ where $p\geq q$ and $p\geq 2$, $p,q \in \mathbb{N}$. If $n $ is sufficiently large then $\mu _{\LL} (n) = n-\lfloor \frac{n}{p+q} \rfloor$.
\end{cor}

Roughly, Theorem~\ref{structure} implies that every $\LL$-free subset of $[n]$ is  `interval like' or `small'.
In the case of sum-free subsets (i.e. when $p=q=1$), a result of Deshouillers, Freiman, S{\'o}s and Temkin~\cite{DFST} provides very precise structural information on the sum-free subsets of $[n]$.
Loosely speaking, they showed that a sum-free subset of $[n]$ is `interval like', `small' or consists entirely of odd numbers.

In the case when $p=q$, Corollary~\ref{cormain} was proven by Hegarty~\cite{hegarty} (without  a lower bound on $n$).

\subsection{The number of solution-free sets}
Write $f (n, \LL)$ for the number of $\LL$-free subsets of $[n]$. In the case when $\mathcal  L$ is $x+y=z$,  define $f (n):=f (n, \LL)$.

By considering all possible subsets of $[n]$ consisting of odd numbers, one observes that there are at least $2^{n/2}$ sum-free subsets of $[n]$. Cameron and Erd\H{o}s~\cite{cam1} conjectured that in fact $f(n)=\Theta (2^{n/2})$. This conjecture was proven independently by Green~\cite{G-CE} and Sapozhenko~\cite{sap}. In fact, they showed that there are constants $C_1$ and $C_2$ such that  $f(n)=(C_i+o(1))2^{n/2}$
for all $n \equiv i \mod 2$.

Results from~\cite{sam, STnew} imply that there are between $2^{(1.16+o(1))\sqrt{n}} $ and $2^{(6.45+o(1))\sqrt{n}} $ Sidon sets in $[n]$. There are also several results concerning the number of so-called $(k,\ell)$-sum-free subsets of $[n]$ (see, e.g.,~\cite{bilu, calk, show}).

More generally, given a linear equation $\LL$, there are at least $2^{\mu _{\LL} (n)}$ $\LL$-free subsets of $[n]$. In light of the situation for sum-free sets one may ask whether, in general, $f(n, \LL)=\Theta(2^{\mu _{\LL} (n)})$. However, Cameron and Erd\H{o}s~\cite{cam1} observed that this is false for translation-invariant $\LL$. In particular, given such an $\LL$-free set, any translation of it is also $\LL$-free.

Green~\cite{G-R} though showed that given a homogeneous
 linear equation $\LL$, $f(n, \LL)=2^{\mu _{\LL} (n)+o(n)}$  (where here the $o(n)$ may depend on $\LL$). 
Our next result implies that one can omit the term $o(n)$ in the exponent for  certain types of linear equation $\LL$.
\begin{thm}\label{number}
Fix $p,q \in \mathbb N$ where    (i) $q\geq 2$ and $p>q(3q-2)/(2q-2)$ or (ii) $q=1$ and $p\geq 3$. Let $\LL$ denote the equation $px+qy=z$. Then
$$f(n, \LL)=\Theta(2^{\mu _{\LL} (n)}).$$
\end{thm}

\subsection{The number of maximal solution-free sets}
Given a linear equation $\LL$,
we say that $S\subseteq [n]$ is a \emph{maximal $\LL$-free subset of $[n]$} if it is $\LL$-free and it is not properly contained in another $\LL$-free subset of $[n]$.
Write $f_{\max} (n, \LL)$ for the number of maximal $\LL$-free subsets of $[n]$. In the case when $\mathcal  L$ is $x+y=z$,  define $f_{\max} (n):=f_{\max} (n, \LL)$.

A significant proportion of the sum-free subsets of $[n]$ lie in just two maximal sum-free sets, namely the set of odd numbers in $[n]$ and the interval $[\lfloor n/2 \rfloor+1,n]$. 
This led Cameron and Erd\H{o}s~\cite{CE} to ask whether $f_{\max} (n) = o(f(n))$ or even $f_{\max} (n) \leq f(n)/ 2^{\eps n}$ for some constant $\eps >0$. \L{u}czak and Schoen~\cite{ls} answered this question in the affirmative, showing that $f_{\max} (n) \leq 2^{n/2-2^{-28}n}$ for sufficiently large $n$. Later, Wolfovitz~\cite{wolf} proved that $f_{\max} (n) \leq 2^{3n/8+o(n)}$. Very recently, Balogh, Liu, Sharifzadeh and Treglown~\cite{BLST, BLST2} proved the following: For each $1\leq i \leq 4$, there is a constant $C_i$ such that, given any $n\equiv i \mod 4$,
$f_{\max }(n)=(C_i+o(1)) 2^{n/4}$.

Except for sum-free sets, the problem of determining the number of maximal solution-free subsets of $[n]$ remains wide open. In this paper we give a number of bounds on $f_{\max} (n, \LL)$ for 
homogeneous linear equations $\LL$ in three variables. The next result gives a general upper bound for such $\LL$. Given a three-variable linear equation $\LL$, an \emph{$\LL$-triple} is a multiset $\{x,y,z\}$ which forms a solution to $\LL$. 
Let $\mu _{\LL }^* (n)$ denote the number of elements $x \in [n]$ that do not lie in \emph{any} $\LL$-triple in $[n]$.
\begin{thm}\label{max1}
Let $\LL$ be a fixed homogenous three-variable linear equation. Then $$f_{\max}(n,\LL) \leq 3^{(\mu_{\LL}(n)-\mu_{\LL}^*(n))/3+o(n)}.$$
\end{thm}
Theorem~\ref{max1} together with the aforementioned result of Green shows that $f_{\max}(n,\LL)$ is significantly smaller than $f(n,\LL)$ for all homogeneous three-variable linear equations $\LL$ that are not translation-invariant. So in this sense it can be viewed
as a generalisation of the result of \L{u}czak and Schoen.
The proof of Theorem~\ref{max1} is a simple application of  container and removal lemmas of Green~\cite{G-R}. The same idea was used to prove results in~\cite{sar, BLST, BLST2}.
Although at first sight the bound in Theorem~\ref{max1} may seem crude, perhaps surprisingly there are equations $\LL$ where the value of $f_{\max}(n,\LL)$ is close to this bound (see Proposition~\ref{MSF2} in Section~\ref{secmax}).

On the other hand, the following result shows that there are linear equations where the bound in Theorem~\ref{max1} is far from tight. 
\begin{thm}\label{max2}
Let $\LL$ denote the equation $px+qy=z$ where $p \geq q \geq 2$ are integers so that $p \leq q^2-q$ and $\gcd (p,q)=q$.  Then $$f_{\max}(n,\LL) \leq 2^{(\mu_{\LL}(n)-\mu_{\LL}^*(n))/2+o(n)}.$$
\end{thm}
In the case when $\LL$ is the equation $2x+2y=z$ we provide a matching lower bound. 
 Again though, we suspect there are equations $\LL$ where the bound in Theorem~\ref{max2} is  far from tight. 
The proof of Theorem~\ref{max2} applies Theorem~\ref{structure} as well as the container and removal lemmas of Green~\cite{G-R}.

We also provide another upper bound on $f_{\max}(n,\LL)$ for a more general class of linear equations. 
\begin{thm}\label{max3}
Let $\LL$ denote the equation $px+qy=z$ where $p\geq q$, $p\geq 2$ and $p,q \in \mathbb{N}$. Then $$f_{\max}(n,\LL) \leq 2^{\mu_{\LL}(\floor{\frac{n-p}{q}})+o(n)}.$$ 
Further, if $q \geq 2$ and $p>q(3q-2)/(2q-2)$ or $q=1$ and $p\geq 3$ then
$$f_{\max}(n,\LL) =O( 2^{\mu_{\LL}(\floor{\frac{n-p}{q}})}).$$ 
\end{thm}
In Section~\ref{secmax} we discuss in what cases a bound as in Theorem~\ref{max3} is stronger than the bound in Theorem~\ref{max2} (and vice versa).
We also provide lower bounds on $f_{\max}(n,\LL)$ for all equations $\LL$ of the form $px+qy=z$ where $p,q\geq 2$ are integers; see Proposition~\ref{MSF6}.

Our results suggest that, in contrast to the case of $f(n,\LL)$, it is unlikely there is a `simple' general asymptotic formula for $f_{\max}(n,\LL)$ for all homogeneous linear equations $\LL$. 
It would be extremely interesting 
to make further progress on this problem.

The paper is organised as follows. In the next section we collect together a number of useful tools. In Section~\ref{SLSF} we prove Theorem~\ref{structure}. Theorem~\ref{number} is proven in Section~\ref{NSFS}.
We prove our results on the number of maximal $\LL$-free sets in Section~\ref{secmax}.



\section{Containers and independent sets in graphs}\label{tools}
\subsection{Container and removal lemmas}
Recently the method of \emph{containers} has proven powerful in tackling a range of problems in combinatorics and other areas,  
 in particular due to the work of Balogh, Morris and Samotij~\cite{BMS} and Saxton and Thomason~\cite{ST}. 
Roughly speaking this method states that for certain (hyper)graphs $G$, the independent sets of $G$ lie only in a small number of subsets of $V(G)$ called \emph{containers}, where each container is an `almost independent set'.

Recall that, given a three-variable linear equation $\LL$, an $\LL$-triple is a multiset $\{x,y,z\}$ which forms a solution to $\LL$. 
Let $H$ denote the hypergraph with vertex set $[n]$ and edges corresponding to $\LL$-triples. Then an independent set in $H$ is precisely an $\LL$-free set.
 
The following container lemma is a special case of a result of Green (Proposition 9.1 of ~\cite{G-R}).
Lemma~\ref{L1}(i)--(iii) is stated explicitly in~\cite{G-R}. Lemma~\ref{L1}(iv)  follows as an immediate consequence of Lemma~\ref{L1}(i) and Lemma~\ref{L2} below.

\begin{lemma}\label{L1}
\cite{G-R} 
Fix a three-variable homogeneous linear equation $\LL$. There exists a family $\mathcal{F}$ of subsets of $[n]$ with the following properties:

\begin{enumerate}[(i)]
  \item{Every $F \in \mathcal{F}$ has at most $o(n^2)$ $\LL$-triples.}
  \item{If $S \subseteq [n]$ is $\LL$-free, then $S$ is a subset of some $F \in \mathcal{F}$.}
  \item{$|\mathcal{F}|=2^{o(n)}$.}
  \item{Every $F \in \mathcal{F}$ has size at most $\mu_\LL(n)+o(n)$.}
\end{enumerate}
\end{lemma}
Throughout the paper we refer to the elements of 
$\mathcal{F}$ as \emph{containers}. Notice that Lemma~\ref{L1}(iv) gives a bound on the size of the containers in terms of $\mu_\LL(n)$ even though, in general, the precise value of $\mu_\LL(n)$ is not known.

The following removal lemma is a special case of a result of Green (Theorem 1.5 in~\cite{G-R}). This result was also generalised to systems of linear equations by Kr\'al', Serra and Vena (Theorem 2 in ~\cite{ksv2}). 

\begin{lemma}\label{L2}
\cite{G-R}  Fix a three-variable homogeneous linear equation $\LL$. Suppose that $A \subseteq [n]$ is a set containing $o(n^2)$ $\LL$-triples. Then there exist $B$ and $C$ such that $A=B \cup C$ where $B$ is $\LL$-free and $|C|=o(n)$.
\end{lemma}

We will also apply the following bound on the number of $\LL$-free sets.

\begin{thm}\label{Gr}
\cite{G-R}  Fix a homogeneous linear equation $\LL$. Then $f(n,\LL) = 2^{\mu_{\LL}(n)+o(n)}$.
\end{thm}

We will use the above results to deduce upper bounds on the number of maximal $\LL$-free sets (Theorems~\ref{max1},~\ref{max2} and~\ref{max3}).

\subsection{Independent sets in graphs}
Let $G$ be a graph and consider any subset $X\subseteq V(G)$. Let IS$(G)$ denote the number of independent sets in $G$. Let $G[X]$ denote the induced subgraph of $G$ on the vertex set $X$ and $G\setminus X$ denote the induced subgraph of $G$ on the vertex set $V(G) \setminus X$.

\begin{fact}\label{ISL}
Let $G$ be a graph and let $A_1 , \dots , A_r$ be a partition of  $V(G)$. Then \rm IS\em $(G) \leq \prod_{i=1}^{r}$ \rm IS\em $(G[A_i])$. 
\end{fact}

The following simple lemma will be used in the proof of Theorem~\ref{number}.
\begin{lemma}\label{ISL2}
Let $G$ be a graph on $n$ vertices and $M$ be a matching in $G$ which consists of  $e$ edges. Suppose that $v\in V(G)$ lies in $M$. Then the number of independent sets in $G$ which contain $v$ is at most $3^{e-1} \cdot 2^{n-2e}$.
\end{lemma}

\begin{proof}
First note that the number of independent sets in $G$ which contain $v$ is at most IS$(G\setminus X)$ where $X$ consists of $v$ and its neighbour in $M$.
Let $A_1 , \dots , A_{e}$ be a partition of the vertex set $V(G\setminus X)$, where if $1 \leq i \leq e-1$ then $A_i$  contains precisely the two vertices from some edge in $M$. So $|A_e|= n-2e$.
Clearly IS$(G[A_i])=3$ for  $1 \leq i \leq e-1$ and IS$(G[A_e])\leq 2^{n-2e}$. The result then follows by Fact~\ref{ISL}.
\end{proof}

\subsection{Link graphs and maximal independent sets}\label{secli}
We obtain many of our results by counting the number of maximal independent sets in various auxiliary graphs. Similar
techniques were used in~\cite{wolf, BLST, BLST2}, and in the graph setting in~\cite{sar, BLPS}. To be more precise, let $B$ and $S$ be disjoint subsets of $[n]$ and fix a three-variable linear equation $\LL$. The \emph{link graph $L_S[B]$ of $S$ on $B$} has vertex set $B$, and an edge set consisting of the following two types of edges:

\begin{enumerate}[(i)]
\item{ Two vertices $x$ and $y$ are adjacent if there exists an element $z \in S$ such that $\{x,y,z\}$ is an $\LL$-triple;}
\item{ There is a loop at a vertex $x$ if there exists an element $z \in S$ or elements $z,z' \in S$ such that  $\{x,x,z\}$ or $\{x,z,z'\}$ is an $\LL$-triple.}
\end{enumerate}
Notice that since the only possible trivial solutions to a three-variable linear equation $\LL$ are of the form $\{x,x,x\}$, all the edges in $L_S[B]$ correspond to non-trivial $\LL$-triples.

 The following simple lemma was  stated in~\cite{BLST,BLST2} for sum-free sets, but  extends to three-variable linear equations. 

\begin{lemma}\label{L4}
Fix a three-variable  linear equation $\LL$. Suppose that $B,S$ are disjoint $\LL$-free subsets of $[n]$. If $I \subseteq B$ is such that $S \cup I$ is a maximal $\LL$-free subset of $[n]$, then $I$ is a maximal independent set in $G:=L_{S}[B]$.
\end{lemma}
Let MIS$(G)$ denote the number of maximal independent sets in $G$. Suppose we have a container $F \in \mathcal F$ as in Lemma~\ref{L1} and suppose $F=A \cup B$ where $B$ is $\LL$-free. Observe that any maximal $\LL$-free subset of $[n]$ in $F$ can be found by first choosing an $\LL$-free set $S \subseteq A$, and then extending $S$ in $B$. Note that by Lemma~\ref{L4}, the number of possible extensions of $S$ in $B$ (which we shall refer to as $N(S,B)$) is bounded from above by the number of maximal independent sets in the link graph $L_S[B]$ (i.e. we have $N(S,B) \leq {\rm MIS}(L_S[B])$). Hence Lemma~\ref{L4} is a useful tool for bounding the number of maximal $\LL$-free subsets of $[n]$.

In particular, we will apply the following result in combination with Lemma~\ref{L4}.
The first part was proven by Moon and Moser~\cite{MM} and the second part by Hujter and Tuza~\cite{HT}. We use the first condition in the proof of   Theorems~\ref{max1} and~\ref{max2}. 

\begin{thm}\label{L5}
Suppose that $G$ is a graph on $n$ vertices possibly with loops. Then the following bounds hold.

\begin{enumerate}[(i)]
\item{ \rm MIS\em $(G) \leq 3^{n/3}$;}
\item{ \rm MIS\em $(G) \leq 2^{n/2}$ if $G$ is additionally triangle-free.}
\end{enumerate}
\end{thm}
To prove Theorem~\ref{max2} we will combine Theorem~\ref{L5}(ii) and the following result.


\begin{lemma}\label{L3}
Let $\LL$ denote the equation $px+qy=z$ where $p\geq q \geq 2$ and $p,q \in \mathbb{N}$. Let $A\subseteq [1,u]$ and let $B \subseteq [u+1,n]$ for some $u \in [n]$.  Consider the link graph $G:=L_{A}[B]$ of $A$ on $B$. If $q^2 \geq p+q$ then $G$ is triangle-free. 
\end{lemma}

\begin{proof} Suppose that $q^2 \geq p+q$ and
suppose for a contradiction there is a triangle in $G$ with vertices $b_1 <b_2<b_3$. By definition of the link graph, there exist $s_1,s_2,s_3 \in A$ such that  $\{b_1,b_2,s_1\}, \{b_2,b_3,s_2\}, \{b_1,b_3,s_3\}$ are $\LL$-triples.



Since all numbers in $A$ are smaller than all numbers in $B$ we have $1 \leq s_1,s_2,s_3 < b_1 < b_2 < b_3$. Also, since $p \geq q \geq 2$, for each of our $\LL$-triples $\{b_i,b_j,s_k\}$ (where $b_i<b_j$) it follows that $b_j$ must play the role of $z$ in $\LL$. 

Define a multiset $\{r_i \in \{p,q\} : 1\leq i \leq 6, r_1\neq r_2, r_3 \neq r_4, r_5 \neq r_6\}$. Consider the three equations $r_1b_1+r_2s_1=b_2, r_3b_2+r_4s_2=b_3$ and $r_5b_1+r_6s_3=b_3$. Combining the second and third gives $b_2=(r_5b_1+r_6s_3-r_4s_2)/r_3$. Then combining this with the first equation gives $(r_1r_3-r_5)b_1 +r_2r_3s_1 +r_4s_2=r_6s_3$. Now since $s_3<b_1$ and all terms are at least $1$, for such an inequality to hold we must have $r_1r_3-r_5<r_6$. Since $r_5 \neq r_6$ this means we have $r_1r_3 < p+q$. Hence as $r_1,r_3 \in \{p,q\}$, in order for $G$ to have a triangle at least one of $p^2 < p+q$, $q^2 < p+q$ and $pq < p+q$ must be satisfied. Since $p \geq q \geq 2$, the first and third are not true and so we must have $q^2 < p+q$, a contradiction.
\end{proof}
\COMMENT{Note: if $q^2 <p+q$ then if $n$ large it seems one can obtain triangles}

We also use link graphs as a means to obtain lower bounds on the number of maximal $\LL$-free sets. We apply the following result in Propositions~\ref{MSF2} and~\ref{MSF6}.

\begin{lemma}\label{L6}
Fix a three-variable  linear equation $\LL$. Suppose that $B,S$ are disjoint $\LL$-free subsets of $[n]$. Let $H$ be an induced subgraph of the link graph $L_S[B]$. Then $f_{\max}(n,\LL) \geq \,$\rm MIS\em $(H)$. 
\end{lemma}
\begin{proof}
Suppose $I$ and $J$ are different maximal independent sets in $H$. First note that $S \cup I$ and $S \cup J$ are $\LL$-free by definition of the link graph. Both cannot lie in the same maximal $\LL$-free subset of $[n]$. To see this, observe by definition of $I$ and $J$, there exists $i \in I \setminus J$. There must exist $s \in S$, $j \in J$ such that $\{i,j,s\}$ forms an $\LL$-triple, else $J \cup \{i\}$ would be an independent set in $H$, which contradicts the maximality of $J$. Hence any maximal $\LL$-free subset of $[n]$ containing $S \cup J$ does not contain $i$. Similarly there exists $j \in J \setminus I$ such that any maximal $\LL$-free subset of $[n]$ containing $S \cup I$ does not contain $j$. The result immediately follows.
\end{proof}

\section{The size of the largest solution-free set}\label{SLSF}
Throughout this section,
 $\LL$ will denote the equation $px+qy=z$ where $p\geq q$ and $p\geq 2$, $p,q \in \mathbb{N}$. 
The aim of this section is to determine the size of the largest $\LL$-free subset of $[n]$. In fact, we will prove a richer structural result on $\LL$-free sets (Theorem~\ref{structure}).
For this, we will introduce the following auxiliary graph $G_m$:
Let $m \in [n]$ be fixed. We define the graph $G_m$ to have vertex set $[m,n]$ and edges between $c$ and $pm+qc$ for all $c \in [m,n]$ such that $pm+qc \leq n$. We will also make use of these auxiliary graphs in Section~\ref{NSFS}.

\begin{fact}\label{F2}{\color{white}.}
\begin{itemize}
\item[(i)]{The size of the largest $\LL$-free subset $S$ of $[n]$ with $\min(S)=m$ is at most the size of the largest independent set in $G_m$ which contains $m$.} 
\item[(ii)]{The number of $\LL$-free subsets $S$ of $[n]$ with $\min(S)=m$ is at most the number of independent sets in $G_m$ which contain $m$.}
\end{itemize}
\end{fact}

\begin{proof}
Let $S$ be an $\LL$-free subset of $[n]$ with $\min(S)=m$. Since $\{m,c,pm+qc\}$ is an $\LL$-triple contained in $[n]$ for all $c \in [m,n]$ such that $pm+qc \leq n$, $S$ cannot contain both $c$ and $pm+qc$. Hence any $\LL$-free subset of $[n]$ with minimum element $m$ is also an independent set in $G_m$ which contains $m$ (although the converse does not necessarily hold). This immediately implies (i) and (ii).
\end{proof}

Note that $G_m$ is a union of disjoint paths (possibly isolated vertices). We refer to the connected components of $G_m$ as the \emph{path components}. Given $G_m$, we define $y_0:=n$, and for $i \geq 1$ define $y_i:= \max \{v \in V(G_m) | \, pm+qv \leq y_{i-1} \}$. Thus we have $y_i=\floor{\frac{y_{i-1}-pm}{q}}$. For $G_m$ we also define $k$ to be the largest $i$ such that $y_i \in [m,n]$, and refer to $k$ as the \emph{path parameter} of $G_m$. We define the \emph{size} of a path component to be the number of \emph{vertices} in it, and we define $N(G_m,i)$ to be the number of path components of size $i$ in $G_m$. 

\begin{fact}\label{F1} 
The graph $G_m$ consists entirely of disjoint path components, where for each $1 \leq i \leq k-1$ there are
$y_{i-1}+y_{i+1}-2y_i$ path components of size $i$, there are $y_{k-1}-2y_k+m-1$ path components of size $k$ and $y_k-m+1$ path components of size $k+1$. 
\end{fact}

\begin{proof}
Every vertex $c \in V(G_m)$ satisfying $y_{j+1} < c \leq y_{j}$ for some $0\leq j\leq k-1$ is in a path in $G_m$ which contains precisely $j$ vertices which are larger than it, whereas every vertex $c > y_j$ is not in such a path. All the vertices in $[m,y_k]$ are in paths which contain precisely $k$ vertices which are larger than it, all vertices in $[y_k+1,y_{k-1}]$ are in paths which contain precisely $k-1$ vertices which are larger than it, and so on. 

Let $A_i$ be the interval $[y_i+1,y_{i-1}]$ for $1\leq i\leq k$ and let $A_{k+1}$ be the interval $[m,y_k]$. There are $|[m,y_k]|=y_k-m+1$ path components of size $k+1$ in $G_m$. For $i\leq k$ all vertices in $A_i$ are the smallest vertex in a path on $i$ vertices, however they may not be the smallest vertex in their path component. In fact, by definition of the $y_i$, all paths which start in $A_j$ for some $j$ must include precisely one vertex from each set $A_{j-1},A_{j-2},\dots,A_1$. This means that for $i\leq k$, the number of path components of size $i$ in $G_m$ is precisely $|A_i|-|A_{i+1}|$. For $i\leq k-1$ this is $y_{i-1}+y_{i+1}-2y_i$ and for $i=k$ this is $y_{k-1}-2y_k+m-1$.
\end{proof}

We now use the graphs $G_m$ and the above facts to obtain the bound for the size of the largest $\LL$-free subset of $[n]$ as stated in Theorem~\ref{structure}.

{\noindent \bf{Proof of Theorem~\ref{structure}.}}
Let $t$ be a non-negative integer. To prove (i) suppose that $t < (\frac{p+q-1}{p+q+p/q})\floor{\frac{n}{p+q}}$. Suppose $S$ is an $\LL$-free set contained in $[\floor{\frac{n}{p+q}}-t,n]$ where $m:=\floor{\frac{n}{p+q}}-t \in S$. By Fact~\ref{F2}(i) we wish to prove that the largest independent set in $G_m$ containing $m$ has size at most $\ceil{\frac{(p+q-1)n}{p+q}}-\floor{\frac{p}{q}t}$. Since $|V(G_m)|=\ceil{\frac{(p+q-1)n}{p+q}}+t+1$ it suffices to show that any independent set $I$ in $G_m$ satisfies $|V(G_m) \setminus I| \geq \floor{(p+q)t/q}+1$.

For $0 \leq i \leq \floor{(p+q)t/q}$, there is an edge between $m+i$ and $(p+q)m+qi$. Note that since $i \leq \floor{(p+q)t/q}$ and $q \leq p$ we have that the largest vertex in any of these edges is indeed at most $n$: 

\begin{center}
$(p+q)(\floor{\frac{n}{p+q}}-t)+qi \leq n-(p+q)t+q\floor{(p+q)t/q} \leq n-(p+q)t+q(p+q)t/q=n$.
\end{center}

Since $I$ can only contain one vertex from each of these edges, we have proven (i), provided that these edges are disjoint. It suffices to show that $\floor{\frac{n}{p+q}}+\floor{pt/q} < (p+q)m=(p+q)(\floor{\frac{n}{p+q}}-t)$ since the left hand side is the largest element of the set $\{m+i : 0 \leq i \leq \floor{(p+q)t/q}\}$. But this immediately follows since $t < (\frac{p+q-1}{p+q+p/q})\floor{\frac{n}{p+q}}$.

To prove (ii) let $t\geq (\frac{p+q-1}{p+q+p/q})\floor{\frac{n}{p+q}}$ and suppose $S$ is an $\LL$-free subset of $[n]$ with $m:=\min(S)=\floor{\frac{n}{p+q}}-t$. By Fact~\ref{F2}(i) $|S|$ is at most the size of the largest independent set in $G_m$ which contains $m$. We will first show that  $G_m$ has path parameter $k\geq 2$, and then the case $q=1$ follows easily. Define $\ell:=\floor{k/2}$ and $$C_k:=\Bigg( \frac{\sum\limits_{i=0}^{2\ell+1} (-1)(-q)^{i}+p\sum\limits_{i=0}^{\ell} q^{2i}}{q^{2\ell+1}+p\sum\limits_{i=0}^{2\ell} q^i} \Bigg).$$ We will show that if $q \geq 2$ then the largest independent set in $G_m$ has size at most $C_k n+k$. We then further bound this from above by $(q^2+1)n/(q^2+q+1)$ for $n$ sufficiently large.
\smallskip

Note that by Fact~\ref{F1}, to prove that $k\geq2$ for $G_m$ it suffices to show that there is a path on 3 vertices in $G_m$. By definition of $k$, $m$ lies on a path $P$ on $k+1$ vertices. Write $P=v_0 v_1 \cdots v_k$ where $m=v_0$ and observe that $v_j=(q^j+p\sum\limits_{i=0}^{j-1} q^i)m$ for $0 \leq j \leq k$. To prove $k\geq 2$ it suffices to show that there is indeed a vertex $(q^2+pq+p)m$ in $V(G_m)$, i.e. $(q^2+pq+p)m \leq n$. Note that since $t\geq (\frac{p+q-1}{p+q+p/q})\floor{\frac{n}{p+q}}$, we have $m = \floor{\frac{n}{p+q}}-t \leq (\frac{p+q+p/q-p-q+1}{p+q+p/q})\floor{\frac{n}{p+q}} = (\frac{p+q}{q^2+pq+p})\floor{\frac{n}{p+q}}$. Hence $(q^2+pq+p)m \leq n$ as desired.
\smallskip

When $q=1$ observe that $y_i=y_{i-1}-pm$, so for $i \leq k-1$ by Fact~\ref{F1} we have $N(G_m,i)=y_{i-1}+y_{i+1}-2y_i=y_i+pm+y_i-pm-2y_i=0$. Hence $G_m$ consists entirely of a union of path components of size either $k$ or $k+1$. Since at most $\ceil{i/2}$ vertices of a path on $i$ vertices can be in an independent set and $k\geq 2$, the largest independent set in $G_m$ has size at most $2n/3=(q^2+1)n/(q^2+q+1)$ in this case, as desired. So now consider the case when $q \geq 2$. We calculate the maximum size of an independent set in $G_m$:

\begin{align*}
 & \sum\limits_{i=1}^{k+1} \ceil{i/2} \cdot N(G_m,i) \\ 
= & \bigg( \sum\limits_{i=1}^{k-1} \ceil{i/2} \cdot (y_{i-1}+y_{i+1}-2y_i) \bigg) + \ceil{k/2} (y_{k-1}+m-1-2y_k) + \ceil{(k+1)/2} (y_k-m+1) \\
= & \, y_0 + \bigg( \sum\limits_{i=1}^{k} (\ceil{(i-1)/2}-2\ceil{i/2}+\ceil{(i+1)/2})y_i \bigg) + (m-1)(\ceil{k/2}-\ceil{(k+1)/2}). \numberthis \label{r1} \\
\end{align*}

Here we used Fact~\ref{F1} in the first equality. For $i$ odd, the coefficient of $y_i$ in (\ref{r1}) is $(i-1)/2-2(i+1)/2+(i+1)/2=-1$. For $i$ even, the coefficient of $y_i$ in (\ref{r1}) is $i/2-2i/2+(i+2)/2=1$. 

The following bounds are obtained from the definition of $y_i$ and $k$: 

\begin{align*}
\rm (a) & \; \; \Big(n-q^j+1-pm\sum\limits_{i=0}^{j-1} q^i\Big)/q^j \leq y_j \leq \Big(n-pm\sum\limits_{i=0}^{j-1} q^i\Big)/q^j; \\
\rm (b) & \; \; n/\Big(q^{k+1}+p\sum\limits_{i=0}^{k} q^i\Big) < m \leq n/\Big(q^{k}+p\sum\limits_{i=0}^{k-1} q^i\Big). \\
\end{align*}\COMMENT{(a) from directly evaluating recursion definition of $\frac{y_{i-1}-pm-q+1}{q}\leq y_i=\floor{\frac{y_{i-1}-pm}{q}} \leq \frac{y_{i-1}-pm}{q}$. (b) $m \leq y_k \leq (n-pm\sum\limits_{i=0}^{k-1} q^i)/q^k$ therefore $(q^k+pm\sum\limits_{i=0}^{k-1} q^i)m \leq n$ giving RHS of (b). If $m \leq n/(q^{k+1}+p\sum\limits_{i=0}^{k} q^i)$ then $y_{k+1} \leq n$ so on a path of length $k+2$, a contradiction.}

Let $\ell:=\floor{k/2}$ (note $k\geq 2$ so $\ell\geq 1$). First suppose $k$ is odd, i.e. $k=2\ell +1$. Using (\ref{r1}), the size of the largest independent set in $G_m$ is bounded above by

\begin{align*}
& y_0 + \Big( \sum\limits_{i=1}^{k} (\ceil{(i-1)/2}-2\ceil{i/2}+\ceil{(i+1)/2})y_i \Big) + (m-1)(\ceil{k/2}-\ceil{(k+1)/2}) \\
= & \, y_0-y_1+y_2-y_3+\cdots+y_{2\ell}-y_{2\ell+1} \\
\stackrel{\rm{(a)}}{\leq} & \, n-\Big(\frac{n-pm-q+1}{q}\Big)+\Big(\frac{n-pm(1+q)}{q^2}\Big)-\Big(\frac{n-pm(1+q+q^2)-q^3+1}{q^3}\Big) \\
& +\cdots-\Bigg(\frac{n-\Big(pm\sum\limits_{i=0}^{2\ell} q^i\Big)-q^{2\ell+1}+1}{q^{2\ell+1}}\Bigg) \\
= & \, n\Big(1-\frac{1}{q}+\frac{1}{q^2}-\cdots-\frac{1}{q^{2\ell+1}}\Big)+m\Big(\frac{p}{q}+\frac{p}{q^3}+\cdots+\frac{p}{q^{2\ell+1}}\Big)+\frac{q-1}{q}+\frac{q^3-1}{q^3} +\cdots+ \frac{q^{2\ell+1}-1}{q^{2\ell+1}} \\
\stackrel{\rm{(b)}}{\leq} & \frac{n}{q^{2\ell+1}} \Big( \sum_{i=0}^{2\ell+1} (-1)(-q)^{i} \Big) +\Bigg(\frac{n}{q^{2\ell+1}+p\sum\limits_{i=0}^{2\ell} q^i}\Bigg)\Bigg(\frac{p\sum\limits_{i=0}^{\ell} q^{2i}}{q^{2\ell+1}}\Bigg)+\frac{k+1}{2} \\
= & \Bigg( \frac{\Big[\sum\limits_{i=0}^{2\ell+1} (-1)(-q)^i \Big] (q^{2\ell+1}+p\sum\limits_{i=0}^{2\ell} q^i)+p\sum\limits_{i=0}^{\ell} q^{2i}}{q^{2\ell+1}(q^{2\ell+1}+p\sum\limits_{i=0}^{2\ell} q^i)} \Bigg)n+\frac{k+1}{2} \\
= & \Bigg( \frac{\sum\limits_{i=0}^{2\ell+1} (-q)^{i+2\ell+1} +p\sum\limits_{i=0}^{\ell} q^{2i+2\ell+1}}{q^{2\ell+1}(q^{2\ell+1}+p\sum\limits_{i=0}^{2\ell} q^i)} \Bigg)n+\frac{k+1}{2} = \Bigg( \frac{\sum\limits_{i=0}^{2\ell+1} (-1)(-q)^{i}+p\sum\limits_{i=0}^{\ell}
q^{2i}}{q^{2\ell+1}+p\sum\limits_{i=0}^{2\ell} q^i} \Bigg) n+\frac{k+1}{2}\\
= & \, C_{k} n+\frac{k+1}{2} \leq C_{k} n+ k. \\
\end{align*}
\COMMENT{Note $p\sum\limits_{i=0}^{2\ell} q^i \sum\limits_{i=0}^{2\ell+1} (-1)(-q)^i=p(1+q+q^2+\cdots+q^{2\ell})(-1+q-q^2+\cdots+q^{2l+1}) =p(-1-q^2-q^4-\cdots-q^{2l}+q^{2l+1}+q^{2l+3}+\cdots+q^{4l+1})=p\sum\limits_{i=0}^{\ell} q^{2i+2l+1}-p\sum\limits_{i=0}^{\ell}q^{2i}$} 
(Note that some of our calculations above did indeed require $q\geq 2$.) By definition, $m \geq y_{k+1}+1$ and for $k$ even, we have $C_k = C_{k+1}$. So if $k$ is even ($k=2\ell$) then we have

\begin{align*}
& y_0 + \Big( \sum\limits_{i=1}^{k} (\ceil{(i-1)/2}-2\ceil{i/2}+\ceil{(i+1)/2})y_i \Big) + (m-1)(\ceil{k/2}-\ceil{(k+1)/2}) \\
= & \, y_0-y_1+y_2-y_3+...+y_{2\ell}-m+1 \leq y_0-y_1+y_2-y_3+...+y_{2\ell}-y_{2\ell+1} \\
\leq & \, C_{k+1} n+\frac{k+2}{2} \leq C_{k} n+k.\\
\end{align*}
The penultimate inequality follows by using calculations from the odd case. The last inequality follows since $k \geq 2$ and $C_k=C_{k+1}$. Thus we have shown that $|S| \leq C_k n+k$ and we know that $k \geq 2$. It remains to show that 

\begin{align}\label{n1}
C_k n+ k \leq \frac{(q^2+1)n}{q^2+q+1}
\end{align}
for $k \geq 2$ and $n$ sufficiently large.

%
We know that $m \leq n/(q^k+p\sum\limits_{i=0}^{k-1} q^i)$ and so $n \geq q^k+p\sum\limits_{i=0}^{k-1} q^i$, therefore condition (\ref{n1}) is met if

\begin{align}
\Big(\frac{q^2+1}{q^2+q+1}-C_k\Big)\Big(q^k+p\sum\limits_{i=0}^{k-1} q^i\Big) \geq k. \label{n2} 
\end{align}

\begin{claim}\label{claimA}
For $k\geq 6$, $(\ref{n2})$ holds.
\end{claim}
Since the proof of Claim~\ref{claimA} is just a technical calculation, we defer it to the appendix.

The claim is not a result which generally holds for $2 \leq k \leq 5$ so instead we directly calculate \COMMENT{Put $k=\{2,3,4,5\}$ into $n\geq k/(\frac{q^2+1}{q^2+q+1}-C_k)$} how large $n$ should be to satisfy (\ref{n1}) in these cases. For $k=3$ and $k=5$ we obtain $n \geq \frac{3(q^3+p(q^2+q+1))(q^2+q+1)}{q^2+1}$ and $n \geq \frac{5(q^5+p(q^4+q^3+q^2+q+1))(q^2+q+1)}{q^4+(p-1)q^3+q^2+1}$ respectively. For $k=2$ and $k=4$ we obtain weaker bounds. Hence taking $n$ to be sufficiently large (larger than these two bounds), we have $C_k n+k \leq \frac{(q^2+1)n}{q^2+q+1}$ for all $k \geq 2$.

\endproof

\section{The number of solution-free sets}\label{NSFS}
Recall  a theorem of Green~\cite{G-R} states that $f(n,\LL) = 2^{\mu_{\LL}(n)+o(n)}$ for any fixed homogeneous linear equation $\LL$.
The aim of this section is to replace the term $o(n)$ here with a constant  for many equations $\LL$. This will be achieved in Theorem~\ref{NSF3}, which immediately implies Theorem~\ref{number}. Denote by $f(n,\LL,m)$ the number of $\LL$-free subsets of $[n]$ with minimum element $m$. We first give bounds on $f(n,\LL,m)$ for linear equations $\LL$ of the form $px+qy=z$.

\begin{lemma}\label{NSF2} 
Let $\LL$ denote the equation $px+qy=z$ where $p\geq q$ and $p\geq 2$, $p,q \in \mathbb{N}$. 

\begin{enumerate}[(i)]
\item{If $m \geq \floor{\frac{n}{p+q}}+1$ then $f(n,\LL,m) = 2^{n-m}$.}
\item{If $m=\floor{\frac{n}{p+q}}$ then $f(n,\LL,m) \leq 2^{\mu_{\LL}(n)-1}$.}

\item{If $q \geq 2$, $m=\floor{\frac{n}{p+q}}-t$ for some positive integer $t$ and $G_m$ has path parameter $1$, then $f(n,\LL,m) \leq 2^{\mu_{\LL}(n)-3/5+t(3q-2p)/(5q)}$.}

\item{If $q \geq 2$, $m=\floor{\frac{n}{p+q}}-t$ for some positive integer $t$ and $G_m$ has path parameter $k \geq 2$, then $f(n,\LL,m) \leq  (4/3) \cdot 2^{(5q^2-2q+2)n/(5q^2)}$.}

\item{If $q=1$, $G_m$ has path parameter $\ell$, and $m=\floor{\frac{n}{\ell p+1}}-t$ for some integer $t$, then $f(n,\LL,m) \leq 2^{(7\ell p+3p)n/(10\ell p+10)+t(7-3p)/10}$. }

\end{enumerate}
\end{lemma}

\begin{proof}
First note that (i) is trivial since all  subsets $S\subseteq [n]$ with $\min (S) \geq \floor{\frac{n}{p+q}}+1$ are $\LL$-free. By Fact~\ref{F2}(ii) we know that $f(n,\LL,m)$ is at most the number of independent sets in $G_m$ which contain $m$. For (ii), there is one edge between $m=\floor{\frac{n}{p+q}}$ and $(p+q)m \leq n$ in $G_m$, hence there are at most $2^{n-\floor{\frac{n}{p+q}}-1}=2^{\mu_{\LL}(n)-1}$ independent sets in $G_m$ containing $m$. 

For (iii) suppose $q \geq 2$ and $m=\floor{\frac{n}{p+q}}-t$ for some  $t\in \mathbb N$. Notice that $G_m$ contains a matching on $y_1-m+1$ edges, namely there is 
an edge between $c$ and $pm+qc$ for $c \in [m,y_1]$. \COMMENT{Crucially these edges are disjoint since $y_2 \leq \frac{y_1-pm}{q} <m$ implies that $y_1 < (p+q)m$.} Observe that $3/4\leq 2^{-2/5}$ and also 
$$y_1-m = \left \lfloor \frac{n-pm}{q} \right \rfloor -m \geq \frac{n-(p+q)m-q}{q} \geq \frac{t(p+q)}{q}-1.$$ 
Hence by Lemma~\ref{ISL2} the total number of independent sets in  $G_m$ which contain $m$ is at most
\begin{align*}
& 2^{n-m-2(y_1-m)-1} 3^{y_1-m} \leq 2^{\mu_{\LL}(n)-1+t} (3/4)^{y_1-m} \\
& \leq   2^{\mu_{\LL}(n)-1+t} (3/4)^{t(p+q)/q-1} \leq 2^{\mu_{\LL}(n)-3/5+t(3q-2p)/(5q)},
\end{align*}
as desired.\COMMENT{Use $3/4=2^{\log_2(3/4)}$ instead: gives slightly better bound $2^{\mu_{\LL}(n)-1-\log_2(3/4)+t(1+\log_2(3/4)(p+q)/q)}$.}

For (iv) suppose $q \geq 2$, $m=\floor{\frac{n}{p+q}}-t$ for some positive integer $t$ and $G_m$ has path parameter $k\geq2$. First note that
\begin{align*}
y_1-y_2 = & \left \lfloor\frac{n-pm}{q}\right \rfloor -\left \lfloor\frac{\floor{\frac{n-pm}{q}}-pm}{q} \right \rfloor \geq \frac{n-pm-q}{q}-\frac{n-pm-qpm}{q^2} \\ 
= & \frac{(q-1)n+pm - q^2}{q^2} \geq \frac{(q-1)n}{q^2}-1. \\
\end{align*}

Define $F(i)$ to be the $i$th Fibonacci number where $F(1)=F(2)=1$. There are $F(i+2)$ independent sets (including the empty set) in a path of length $i$. Recall the following Fibonacci identity: $F(i+2)F(i)-F(i+1)^2=(-1)^{i+1}$. If $i$ is even and $a>b$ then 
$$ \bigg( \frac{F(i)F(i+2)}{F(i+1)^2} \bigg)^a \bigg( \frac{F(i+1)F(i+3)}{F(i+2)^2} \bigg)^b =  \bigg( \frac{F(i+1)^2-1}{F(i+1)^2} \bigg)^a \bigg( \frac{F(i+2)^2+1}{F(i+2)^2} \bigg)^b \leq 1.$$ 
\COMMENT{Induction}
Also observe that by omitting $(F(i+1)F(i+3)/F(i+2)^2)^b$ the inequality still holds. By use of Fact~\ref{F1} and applying the above bounds, we can bound from above the number of independent sets in $G_m$ as required:

\begin{align*}
& 2^{y_0+y_2-2y_1} 3^{y_1+y_3-2y_2} 5^{y_2+y_4-2y_3} \dots F(k+1)^{y_{k-2}+y_k-2y_{k-1}} F(k+2)^{y_{k-1}+m-2y_k-1} F(k+3)^{y_k-m+1} \\
= & \, 2^{y_0+y_2-2y_1} 3^{y_1-2y_2} 5^{y_2} \bigg(\frac{3\cdot 8}{5^2}\bigg)^{y_3} \bigg(\frac{5\cdot 13}{8^2}\bigg)^{y_4} \cdots \bigg(\frac{F(k+1)\cdot F(k+3)}{F(k+2)^2}\bigg)^{y_k} \bigg(\frac{F(k+2)}{F(k+3)}\bigg)^{m-1} \\ 
\leq & \, 2^{y_0+y_2-2y_1} 3^{y_1-2y_2} 5^{y_2} \leq 2^{y_0+y_2-2y_1+y_2} 3^{y_1-y_2} = 2^{y_0}(3/4)^{y_1-y_2} \leq 2^{n}(3/4)^{(q-1)n/q^2-1} \\
\leq & \, (4/3) \cdot 2^{n-2(q-1)n/(5q^2)} = (4/3) \cdot 2^{(5q^2-2q+2)n/(5q^2)}.
\end{align*}

For (v), since $y_i=n-ipm$ Fact~\ref{F1} implies that if $G_m$ has path parameter $\ell$, then $G_m$ is a union of paths of length $\ell$ and $\ell+1$. We use the bound $F(i) \leq 2^{(7i-11)/10}$ (a simple proof by induction which holds for $i\geq2$). Since $m<y_\ell=n-\ell pm$ we can write $m=\floor{\frac{n}{\ell p+1}}-t$ for some integer $t\geq 0$. Now using these bounds, we have

\begin{align*}
& F(\ell+2)^{y_{\ell-1}-2y_{\ell}+m} F(\ell+3)^{y_{\ell}-m} = F(\ell+2)^{(\ell p+p+1)m-n} F(\ell+3)^{n-(\ell p+1)m} \\
\leq & \, 2^{(3+7\ell)((\ell p+p+1)m-n)/10+(10+7\ell)(n-(\ell p+1)m)/10} = 2^{(7n+(3p-7)m)/10} \\
\leq & \, 2^{(7n+(3p-7)(n/(\ell p+1)-t))/10} = 2^{(7\ell p+3p)n/(10\ell p+10)+t(7-3p)/10}. 
\end{align*}
\end{proof}

\begin{thm}\label{NSF3} 
Let $\LL$ denote the equation $px+qy=z$ where $p,q \in \mathbb{N}$ and 
\begin{itemize}
\item[(i)] $q\geq 2$ and $p>q(3q-2)/(2q-2)$ or;
\item[(ii)] $q=1$ and $p\geq 3$. 
\end{itemize}
Then $f(n,\LL) \leq (3/2+o(1)+C) 2^{\mu_{\LL}(n)}$ where for (i) $C:=\frac{2^{-2p/(5q)}}{1-2^{(3q-2p)/(5q)}}$ and for (ii) $C:=\frac{2^{(7-3p)/10}}{1-2^{(7-3p)/10}}$.
\end{thm}

\begin{proof} For both cases by Lemma~\ref{NSF2}(i)--(ii) there are at most $3\cdot 2 ^{\mu _{\LL} (n) -1}$ $\LL$-free subsets $S$ of $[n]$ where $\min (S) \geq \floor{\frac{n}{p+q}}$. 
For (i), first consider $\LL$-free subsets arising from Lemma~\ref{NSF2}(iv). Since $k \geq 2$, 
$$m < y_2 = \bigg \lfloor \frac {\floor{\frac{n-pm}{q}}-pm}{q} \bigg \rfloor \leq \frac{n-pm-qpm}{q^2}$$ 
and so $m \leq n/(q^2+pq+p)$. Now as $n \rightarrow \infty$, $$ \frac{n/(q^2+pq+p) \cdot (4/3) \cdot 2^{(5q^2-2q+2)n/(5q^2)}}{2^{\mu_{\LL}(n)}} = \frac{2^{\log_2(4n/(3(q^2+pq+p)))+(5q^2-2q+2)n/(5q^2)}}{2^{\mu_{\LL}(n)}} \rightarrow 0,$$ as long as we have $2^{(5q^2-2q+2)n/(5q^2)} \ll 2^{\mu_{\LL}(n)}$. This is satisfied if $(5q^2-2q+2)/(5q^2) < (p+q-1)/(p+q)$ which when rearranged, gives $p>q(3q-2)/(2q-2)$. 

For $\LL$-free subsets arising from Lemma~\ref{NSF2}(iii), set $a:=2^{\mu_{\LL}(n)-3/5}$, $r:=2^{(3q-2p)/(5q)}$ and let $u$ be the largest $t$ such that $G_m$ with $m=\floor{\frac{n}{p+q}}-t$ has path parameter $1$. Then since  $p>q(3q-2)/(2q-2)>3q/2$ we have $|r|<1$ and so 
$$ \sum\limits ^{u}_{t=1} 2^{\mu_{\LL}(n)-3/5+t(3q-2p)/(5q)} \leq \sum\limits_{t=1}^{\infty} ar^t= \sum\limits_{t=0}^{\infty} (ar)r^t= \frac{ar}{1-r} = \frac{2^{\mu_{\LL}(n)-2p/(5q)}}{1-2^{(3q-2p)/(5q)}}.$$
Altogether this implies that $f(n,\LL) \leq (3/2+o(1)+C) 2^{\mu_{\LL}(n)}$ where $C:=\frac{2^{-2p/(5q)}}{1-2^{(3q-2p)/(5q)}}$.

For (ii), set $a:=2^{(7kp+3p)n/(10kp+10)}$, set $r:=2^{(7-3p)/10}$ and let $u$ be the largest $t$ such that $G_m$ with $m:=\floor{\frac{n}{p+q}}-t$ has path parameter $k$ for any fixed $k \in \mathbb N$. Since $p \geq 3$ we have $|r|<1$ and so 

$$ \sum\limits^{u}_{t=1} 2^{(7kp+3p)n/(10kp+10)+t(7-3p)/10} \leq \sum\limits_{t=1}^{\infty} ar^t= \sum\limits_{t=0}^{\infty} (ar)r^t= \frac{ar}{1-r} = \frac{2^{(7kp+3p)n/(10kp+10)+(7-3p)/10}}{1-2^{(7-3p)/10}}.$$

For $k=1$ the last term is at most $2^{(\mu_{\LL}(n)+(7-3p)/10)}/(1-2^{(7-3p)/10})$. For $k\geq 2$ we obtain a term which is $o(2^{\mu_{\LL}(n)})$ as $n$ tends to infinity, since $(7kp+3p)n/(10kp+10) < \mu_{\LL}(n)$ for $p\geq 3$. Therefore, Lemma~\ref{NSF2} implies that $f(n,\LL) \leq (3/2+o(1)+C) 2^{\mu_{\LL}(n)}$ where $C:=\frac{2^{(7-3p)/10}}{1-2^{(7-3p)/10}}$.

\end{proof}


\section{The number of maximal solution-free sets}\label{secmax}

\subsection{A general upper bound}
Let $\LL$ be a three-variable linear equation.  Let $\mathcal{M}_{\LL} (n)$ denote the set of elements $x \in [n]$ such that $x \in [n]$ does not lie in \emph{any} $\LL$-triple in $[n]$. Define $\mu_{\LL}^*(n):=|\mathcal{M}_{\LL}(n)|$.
For example, if $\LL$ is translation-invariant then $\{x,x,x\}$ is an $\LL$-triple for all $x \in [n]$  so $\mathcal{M}_{\LL} (n)=\emptyset$ and $\mu_{\LL}^*(n)=0$.

Let $\LL$ denote the equation $px+qy=z$ where $p\geq 2$, $p \geq q$ and $p,q \in \mathbb{N}$. Write $u:=\gcd(p,q)$. Then notice that $\mathcal{M}_{\LL}(n)\supseteq \{s \in [n]: s>\floor{(n-p)/q}, u \nmid s\}$.
This follows since if $s> \floor{(n-p)/q}$ then $ps+q \geq qs+p > n$ and so $s$ cannot play the role of $x$ or $y$ in an $\LL$-triple in $[n]$. If $u \nmid s$ then as $u|(px+qy)$ for any $x,y \in[n]$ we have that $s$ cannot play the role of
$z$  in an $\LL$-triple in $[n]$. Actually, for large enough $n$ we  have $\mathcal{M}_{\LL}(n)=\{s: s>\floor{(n-p)/q}, u \nmid s\}$ for all such $\LL$. We omit the proof of this here.

\COMMENT{If $\gcd(p,q)=1$, then the largest $s$ such that there does not exist $a,b$ non-negative integers such that $pa+qb=s$ is $pq-(p+q)$. 
(J.J. Sylvester, Mathematical questions with their solutions, {\em Educational Times}, 41 (21), (1884).)
\\
Note that if $s>pq+p$ then certainly $pa+qb=s$ for non-negative $a,b$. Suppose $a=0$. Then $pq+q(b-p)=s$. So provided $b>p$ we are happy. But that is true by choice of $s$.
\\
 If $\gcd(p,q)=u$, the provided $s>(pq+p)/u$ we have that $pa/u+qb/u=s$. That is, for large enough $s$ where $u|s$ we can write $s$ in form $s=pa'+qb'$ for positive $a',b'$.
\\
Suppose $\gcd(p,q)=u$, $s>\floor{(n-p)/q}$ and $u$ divides $s$. If $n$ is sufficiently large, $s$ will be sufficiently large. Hence, we can write $s=pa'+qb'$ for positive $a',b'$. Hence $s$ is not in $\mathcal M_{\LL} (n)$ as desired.}

We now prove Theorem~\ref{max1}.

\begin{thm-hand}[\ref{max1}.]
Let $\LL$ be a fixed homogenous three-variable linear equation.  Then $$f_{\max}(n,\LL) \leq 3^{(\mu_{\LL}(n)-\mu_{\LL}^*(n))/3+o(n)}.$$
\end{thm-hand}
\begin{proof}
Let $\mathcal{F}$ denote the set of containers obtained by applying Lemma~\ref{L1}. Since every $\LL$-free subset of $[n]$ lies in at least one of the $2^{o(n)}$ containers, it suffices to show that every $F \in \mathcal F$ houses at most $3^{(\mu_\LL(n)-\mu_{\LL}^*(n))/3+o(n)}$ maximal $\LL$-free subsets. 

Let $F \in \mathcal F$. By Lemmas~\ref{L1}(i) and~\ref{L2}, $F=A \cup B$ where $|A|=o(n)$, $|B| \leq \mu_\LL(n)$ and $B$ is $\LL$-free. Note that we can add all the elements of  $\mathcal{M}_\LL(n)$ to $B$ (and thus $ F$) whilst ensuring that $|B| \leq \mu_\LL(n)$ and $B$ is $\LL$-free. So we may assume that  $\mathcal{M}_\LL(n) \subseteq B$.

 Each maximal $\LL$-free subset of $[n]$ in $F$ can be found by picking a subset $S \subseteq A$ which is $\LL$-free, and extending it in $B$. The number of ways of doing this is the number of ways of choosing the subset $S$ multiplied by the number of ways of extending a fixed $S$ in $B$, which we denote by $N(S,B)$. Since $|A|=o(n)$, there are $2^{o(n)}$ choices for $S$.
It therefore suffices to show that for any $S \subseteq A$, we have $N(S,B) \leq 3^{(\mu_\LL(n)-\mu_{\LL}^*(n))/3}$. 

Consider the link graph $G:=L_S[B]$. Then by definition, $\mathcal M_{\LL} (n)$ is an independent set in $G$. Thus, ${\rm MIS} (G)= {\rm MIS}(G\setminus \mathcal{M}_\LL (n))$.
Further, 
 Lemma~\ref{L4} and Theorem~\ref{L5}(i) imply that 
 $$N(S,B) \leq {\rm  MIS}(G) =  {\rm MIS}(G\setminus \mathcal{M}_\LL (n)) \leq 3^{|B\setminus \mathcal{M}_\LL(n)|/3} \leq 3^{(\mu_\LL(n)-\mu_{\LL}^*(n))/3},$$ as desired.
\end{proof}
As mentioned in the introduction, Theorem~\ref{max1} together with Theorem~\ref{Gr} shows that $f_{\max}(n,\LL)$ is significantly smaller than $f(n,\LL)$ for all homogeneous three-variable linear equations $\LL$ that are not translation-invariant. So in this sense it can be viewed as a generalisation of a result of \L{u}czak and Schoen~\cite{ls} on sum-free sets.

Let $\LL$ denote the equation $px+y=z$ for some $p \in \mathbb N$. Notice that in this case we have $\mu ^* _{\LL} (n)=0$ for $n >p$. The next result implies that if $p$ is large then $f_{\max}(n,\LL)$ is close to the bound in Theorem~\ref{max1}. So for such equations $\LL$, Theorem~\ref{max1} is close to best possible.
\COMMENT{AT: prop implies that given any $\eps >0$, if $p$ sufficiently large then have $f_{\max}(n,\LL) \geq 3^{(1+\eps)\mu_{\LL} (n)/3}$}
\begin{prop}\label{MSF2}
Given $p,n \in \mathbb N$ where $p \geq 2$, let $\LL$ denote the equation $px+y=z$. Then $$f_{\max}(n,\LL) \geq 3^{\mu_{\LL} (n)/3-2pn/(3(p+1)(3p^2-1))-p-5}.$$ 
\end{prop}
\proof
Given $p,n \in \mathbb N$, let $\LL$ denote the equation $px+y=z$. Set $s:=\floor{\frac{(p-1)n}{3p^2-1}}$ and $a:=\floor{\frac{n-s}{p}}$. Consider the link graph $G:=L_{\{s,2s\}}[a+1,a+3ps]$. Observe that:
\begin{align*}
& 2s \leq \frac{(2p-2)n}{3p^2-1} < \frac{n}{p+1} < \frac{(3p-1)n}{3p^2-1} = \frac{n}{p}-\frac{(p-1)n}{3p^3-p} \leq \frac{n-s}{p} < a+1; \\
& a+3ps = \left \lfloor \frac{n-s}{p} \right \rfloor +3ps \leq \frac{n}{p}+\bigg(3p-\frac{1}{p}\bigg)s = \frac{n}{p} + \frac{3p^2-1}{p} \left \lfloor \frac{(p-1)n}{3p^2-1} \right \rfloor \leq \frac{n+n(p-1)}{p} =n.
\end{align*}

As a consequence, the sets $\{s,2s\}$ and $[a+1,a+3ps]$ (a subset of $[\floor{\frac{n}{p+1}}+1,n]$) are disjoint $\LL$-free sets in $[n]$, and so Lemma~\ref{L6} implies that $f_{\max}(n,\LL) \geq \,$\rm MIS\em $(G)$. \rm It remains to show that $G$ contains at least $3^{\mu_{\LL} (n)/3-2pn/(3(p+1)(3p^2-1))-6}$ maximal independent sets. 

Observe that for each $i \in [ps]$ there is an edge in $G$ between $a+i$ and $a+ps+i$ (since $\{s,a+i,a+i+ps\}$ is an $\LL$-triple), an edge between $a+i+ps$ and $a+i+2ps$ (since $\{s,a+i+ps,a+i+2ps\}$ is an $\LL$-triple) and an edge between $a+i$ and $a+i+2ps$ (since $\{2s,a+i,a+i+2ps\}$ is an $\LL$-triple). Also since $a>(n-s)/p-1$, we have $p(a+1)+s>n$ and hence there are no further edges in $G$.

Hence $G$ is a collection of $ps$ disjoint triangles, where 4 vertices in $G$ have loops ($(p+1)s$, $(p+2)s$, $(2p+1)s$ and $(2p+2)s$). So $G$ has at least $3^{ps-4}$ maximal independent sets. Now observe:

\begin{align*}
ps-4-\frac{\mu_\LL (n)}{3} = & \, p \left \lfloor \frac{(p-1)n}{3p^2-1} \right \rfloor -4 - \frac{n}{3} + \frac{1}{3} \left \lfloor \frac{n}{p+1} \right \rfloor \geq \bigg( \frac{p^2-p}{3p^2-1} - \frac{1}{3}+\frac{1}{3(p+1)} \bigg) n -p-5 \\
= & \bigg(\frac{-2p}{3(p+1)(3p^2-1)} \bigg)n-p-5, 
\end{align*}
as required.
\qed

\subsection{Upper bounds for $px+qy=z$}

Let $\LL$ denote the equation $px+qy=z$ where $p\geq q$, $p\geq 2$ and $p,q \in \mathbb{N}$. For such $\LL$, the next simple result provides an alternative bound to Theorem~\ref{max1}.
\begin{lemma}\label{MSF3}
Let $\LL$ denote the equation $px+qy=z$ where $p\geq q$, $p\geq 2$ and $p,q \in \mathbb{N}$. Then $f_{\max}(n,\LL) \leq f(\floor{(n-p)/q},\LL)$.
\end{lemma}

\begin{proof}
Set $C:=[\floor{\frac{n-p}{q}}]$ and $B:=[\floor{\frac{n-p}{q}}+1,n]$. In particular, $B$ is $\LL$-free. Notice that every maximal $\LL$-free subset of $[n]$ can be found by selecting an $\LL$-free subset $S \subseteq C$ and then extending it in $B$ to a maximal one. Suppose we have such an $\LL$-free subset $S$. 
By Lemma~\ref{L4}, the number of such extensions of $S$ is at most ${\rm{MIS}}(L_S[B])$.

For any $\LL$-triple $\{x,y,z\}$ in $[n]$ satisfying $px+qy=z$, since $z \leq n$, we must have $x \leq \frac{n-q}{p}$ and $y \leq \frac{n-p}{q}$. Hence $x,y \in C$. This means that there are no $\LL$-triples in $[n]$ which contain more than one element from $B$. Thus the link graph $L_S[B]$ must only contain isolated vertices and loops. So $L_S[B]$ has precisely one maximal independent set. Hence the number of maximal $\LL$-free subsets of $[n]$ is bounded by the number of choices of $S$ in $C$ which are $\LL$-free, i.e. 
$f(\floor{(n-p)/q},\LL)$.
\end{proof}
Lemma~\ref{MSF3} together with Theorems~\ref{number} and~\ref{Gr} immediately implies Theorem~\ref{max3}.


The next result gives a further upper bound on $f_{\max}(n,\LL)$ for certain linear equations $\LL$. Notice that for such $\LL$, Theorem~\ref{max2} yields a better bound than Theorem~\ref{max1}.

\begin{thm-hand}[\ref{max2}.]
Let $\LL$ denote the equation $px+qy=z$ where $p \geq q \geq 2$ are integers so that $p \leq q^2-q$ and $\gcd (p,q)=q$.  Then $$f_{\max}(n,\LL) \leq 2^{(\mu_{\LL}(n)-\mu_{\LL}^*(n))/2+o(n)}.$$
\end{thm-hand}

\begin{proof}
 Let $\mathcal{F}$ denote the set of containers obtained by applying Lemma~\ref{L1}. Since every $\LL$-free subset of $[n]$ lies in at least one of the $2^{o(n)}$ containers, it suffices to show that every $F \in \mathcal F$ houses at most $2^{(\mu_{\LL}(n)-\mu_{\LL}^*(n))/2+o(n)}$ $\LL$-free sets. 

Let $F \in \mathcal F$. By Lemmas~\ref{L1}(i) and~\ref{L2}, $F=A \cup B$ where $|A|=o(n)$, $|B| \leq \mu_\LL(n)$ and $B$ is $\LL$-free. Note that we can add all the elements of  $\mathcal{M}_\LL(n)$ to $B$ (and thus $ F$) whilst ensuring that $|B| \leq \mu_\LL(n)$ and $B$ is $\LL$-free. So we may assume that  $\mathcal{M}_\LL(n) \subseteq B$.
By Theorem~\ref{structure},  $\min(B)=\floor{\frac{n}{p+q}}-t$ for some non-negative integer $t < (\frac{p+q-1}{p+q+p/q})\floor{\frac{n}{p+q}}$ and $|B| \leq \ceil{\frac{(p+q-1)n}{p+q}}-\floor{\frac{p}{q}t}$, or $|B| \leq \frac{(q^2+1)n}{q^2+q+1}$. 

\smallskip

\noindent
{\bf Case 1}: $\min(B)=\floor{\frac{n}{p+q}}-t$ for $0 \leq t < (\frac{p+q-1}{p+q+p/q})\floor{\frac{n}{p+q}}$. Write $F=X \cup Y$ where $Y \subseteq [\floor{\frac{n}{p+q}}+1,n]$ is $\LL$-free, and $X \subseteq [1,\floor{\frac{n}{p+q}}]$. 
Note that $|X|=t'+o(n)$  and $|Y| \leq \ceil{\frac{(p+q-1)n}{p+q}}-\floor{\frac{p}{q}t}-t'+o(n)$ where $t' \leq t$. Also $\mathcal M_{\LL} (n) \subseteq Y$. Choose $S \subseteq X$ to be $\LL$-free. Consider the link graph $L_{S}[Y]$ and observe that by Lemma~\ref{L4}, $N(S,Y)\leq {\rm MIS}(L_S[Y])$.
(Recall $N(S,Y)$ denotes the number of extensions of $S$ in $Y$ to a maximal $\LL$-free set.)

Since $p \leq q^2-q$, by Lemma~\ref{L3} $L_{S}[Y]$ is triangle-free. 
By definition, $\mathcal M_{\LL} (n)$ is an independent set in $L_S[Y]$ and so  MIS($L_{S}[Y]$)$=$MIS($L_{S}[Y\setminus \mathcal{M}_{\LL}(n)]$). Therefore Theorem~\ref{L5}(ii) implies that MIS($L_{S}[Y]$)$ \leq 2^{(|Y|-|\mathcal{M}_{\LL} (n)|)/2}$. Overall, this implies that the number of $\LL$-free sets  contained in $F$  is at most $$2^{|X|} \times 2^{(|Y|-|\mathcal{M}_{\LL}(n)|)/2} \leq 2^{t'+o(n)+(\mu_{\LL}(n)-\mu_{\LL}^*(n)-\floor{\frac{p}{q}t}-t')/2} \leq 2^{(\mu_{\LL}(n)-\mu_{\LL}^*(n))/2+o(n)},$$
as desired. 

\smallskip

\noindent
{\bf Case 2}: $|B| \leq \frac{(q^2+1)n}{q^2+q+1}$. In this case $|F| \leq \frac{(q^2+1)n}{q^2+q+1}+o(n)$. Choose any $\LL$-free $S \subseteq A$ (note there are at most $2^{o(n)}$ choices for $S$). 
Consider the link graph $L_{S}[B]$ and observe by Lemma~\ref{L4} that $N(S,B)\leq {\rm MIS}(L_S[B])$. Similarly as in Case~1 we have that
 MIS($L_{S}[B]$)$=$MIS($L_{S}[B']$) where $B':=B\setminus \mathcal{M}_{\LL}(n)$. By Theorem~\ref{L5}(i),

$${\rm{MIS}} {(L_{S}[B']) \leq 3^{|B'|/3} \leq 3^{((q^2+1)n/(3(q^2+q+1))-\mu_{\LL}^*(n)/3)} \leq 2^{(\mu_{\LL}(n)-\mu_{\LL}^*(n))/2+o(n)}}.$$
The last inequality  follows since $\mu_{\LL} (n)=n-\floor{n/(p+q)}$ and
 $\mathcal{M}_{\LL}(n)=\{s: s>\floor{(n-p)/q}, q \nmid s\}$ since $\gcd(p,q)=q$.

To see this, first note that $$\mu^*_{\LL}(n) = \frac{(q-1)^2 n}{q^2}-o(n).$$ Hence for the inequality to hold we require that $$9^{((q^2+1)/(q^2+q+1)-(q^2-2q+1)/(q^2))} < 8^{((p+q-1)/(p+q)-(q^2-2q+1)/(q^2))}.$$ Let $a:=\log_{9}8$. This rearranges to give $$p> \frac{(1-a)(q^4-q)+q^3+q^2}{(2a-1)q^3+(a-1)(q^2+q-1)}.$$ Since $p \geq q$ it suffices to show that $(3a-2)q^3+(a-2)(q^2+q)+(2-2a)>0$. This indeed holds since $q \geq 2$. 

Overall, this implies that the number of $\LL$-free sets  contained in $F$  is at most $2^{(\mu_{\LL}(n)-\mu_{\LL}^*(n))/2+o(n)},$ as desired.
\end{proof}
The proof of Theorem~\ref{max2} actually generalises to some other equations $px+qy=z$ where $\gcd(p,q) \not = q$ (but still $p \leq q^2-q$). However, in these cases Theorem~\ref{max3} produces a better upper bound on $f_{\max} (n, \LL)$. The next result summarises when
Theorem~\ref{max1},~\ref{max2} or~\ref{max3} yields the best upper bound on $f_{\max} (n, \LL)$. We defer the proof to the appendix.

\begin{prop}\label{best}
Let $\LL$ denote the equation $px+qy=z$ where $p\geq q$, $p\geq 2$ and $p,q \in \mathbb{N}$. 
Up to the error term in the exponent, the best upper bound on $f_{\max}(n,\LL)$ given by Theorems~\ref{max1},~\ref{max2} and~\ref{max3}  is:

\begin{enumerate}[(i)]
\item{$f_{\max}(n,\LL) \leq 3^{(\mu_{\LL}(n)-\mu_{\LL}^*(n))/3+o(n)}$ if $\gcd(p,q)=q$, $p\geq q^2$, and either $q\leq 9$ or $10\leq q \leq 17$ and $p<(a-1)(q^2-q)/(q(2-a)-1)$ where $a:=\log_3(8)$;}
\item{$f_{\max}(n,\LL) \leq 2^{(\mu_{\LL}(n)-\mu_{\LL}^*(n))/2+o(n)}$ if $\gcd(p,q)=q$ and $p \leq q^2-q$;}
\item{$f_{\max}(n,\LL) \leq 2^{\mu_{\LL}(\floor{(n-p)/q})+o(n)}$ otherwise.}
\end{enumerate}
\end{prop}

\subsection{Lower bounds for $px+qy=z$}

The following result provides lower bounds on $f_{\max}(n,\LL)$ for all equations $\LL$ of the form $px+qy=z$ where $p\geq q \geq 2$.

\begin{prop}\label{MSF6}
Let $\LL$ denote the equation $px+qy=z$ where $p\geq q \geq 2$ are integers. Suppose that $n> 2p$. In each case $f_{\max}(n,\LL) \geq 2^{\ell}$ where $\ell$ is defined as follows: 
\begin{enumerate}[(i)]
\item{$\ell:=(n(q-1)-pq+q-2q^2)/q^2$ if $p\geq q^2$,}
\item{$\ell:=(n(p-q)-p^2+q^2-2pq)/(pq)$ if $q<p<q^2$,}
\item{$\ell:=(n-6q)/2q$ if $p=q$.}
\end{enumerate}
\end{prop}

\begin{proof}
For each case, we shall let $B:=[\floor{\frac{n}{p+q}}+1,n]$, and consider the link graph $G:=L_{\{1\}}[B]$. Since $B$ and $\{1\}$ are $\LL$-free, by Lemma~\ref{L6} it suffices to show that there is an induced subgraph of $G$ which contains at least $2^{\ell}$ maximal independent sets. For each case we will find an induced perfect matching on $2\ell$ vertices in $G$. (Note there are $2^{\ell}$ maximal independent sets in such a matching.) 

More specifically, for each case we shall find an interval $I:=[a,b]$ for some $a,b \in V(G)$ and let $J:=\{qi+p | \, i \in I\}$. Note that all edges in $G$ (other than at most one loop) are of the form $\{i,qi+p\}$ and $\{i,pi+q\}$. By our choice of $I$ and $J$,  $G[I \cup J]$ will form a perfect matching on $2|I|$ vertices if the following conditions hold:

\begin{enumerate}[(1)]
\item{$qa+p>b$ (which ensures that $I \cap J = \emptyset$),}
\item{$qb+p\leq n$ (which ensures that $J\subseteq [n]$),}
\item{$pa+q>n$ (which ensures that the only edges in $G$ are of the form $\{i,qi+p\}$),}
\item{$p+q<a$ (which ensures that there is no loop at a vertex in $G[I\cup J]$).}
\end{enumerate}
Notice that actually we do not require condition (3) to hold in the case when $p=q$. Indeed, this is because in this case an edge $\{i,pi+q\}$ in $G$ is the same as the edge $\{i,qi+p\}$.
Further, there is at most one loop in $G$ (if $p+q\in B$). So even if (4) does not hold we will obtain an induced matching in $G$ on $2|I|-2$ vertices.

Thus, to obtain an induced matching in $G$ on $2|I|-2$ vertices it suffices to choose $a$ and $b$ so that (1)--(3) hold except when $p=q$ when we only require that (1) and (2) hold.

By choosing $b:=\floor{(n-p)/q}$, (2) holds since $qb+p=q\floor{(n-p)/q}+p \leq q(n-p)/q+p=n$. 

If $p\geq q^2$ then set $a:=\floor{(n-q)/q^2}+1$. Then $a \in B$ and
further $pa+q \geq q^2 a+q > q^2 ((n-q)/q^2) +q =n$ and $qa+p \geq qa+q^2 > q((n-q)/q^2)+q^2=n/q-1+q^2 > \floor{(n-p)/q}=b$. So (1) and (3) hold.

If $q<p<q^2$ then set $a:=\floor{(n-q)/p}+1$. So $a \in B$. Further, $pa+q > p((n-q)/p)+q=n$ and $qa+p > q((n-q)/p)+p = qn/p -q^2/p +p > qn/q^2-q+p > n/q > \floor{(n-p)/q}=b$. 
So (1) and (3) hold.

If $p=q$ set $a:=\floor{n/(p+q)}+1=\floor{n/(2q)}+1 \in B$. Observe that $qa+q>qn/2q+q>n/2> \floor{(n-q)/q}=b$ since $q \geq 2$. So (1) holds.

Now calculating the size of the interval $I=[a,b]$ in each case proves the result:

\begin{itemize}
\item{If $a=\floor{(n-q)/q^2}+1$, then $|I|-1=\floor{(n-p)/q}-(\floor{(n-q)/q^2}+1) \geq (n-p)/q-1-(n-q)/q^2-1 = (n(q-1)-pq+q-2q^2)/q^2$.}
\item{If $a=\floor{(n-q)/p}+1$, then $|I|-1=\floor{(n-p)/q}-(\floor{(n-q)/p}+1) \geq (n-p)/q-1-(n-q)/p-1 = (n(p-q)-p^2+q^2-2pq)/(pq)$.}
\item{If $a=\floor{n/(p+q)}+1$ then $|I|-1=\floor{(n-p)/q}-(\floor{n/(p+q)}+1) \geq (n-p)/q-1-n/(p+q)-1 = (pn-(p+2q)(p+q))/(q(p+q)) = (qn-6q^2)/(2q^2)=(n-6q)/2q$.}
\end{itemize}
\end{proof}

Although the lower bounds in Proposition~\ref{MSF6} do not meet the upper bounds in Theorems~\ref{max2} and~\ref{max3} in general,
 Theorem~\ref{max2} and Proposition~\ref{MSF6}(iii) do immediately imply the following asymptotically exact result.

\begin{thm}\label{MSF7}
Let $\LL$ denote the equation $2x+2y=z$. Then $f_{\max}(n,\LL)=2^{n/4+o(n)}$. 
\end{thm}
Since submitting this paper, we have also given a general upper bound on $f_{\max}(n,\LL)$ for equations $\LL$ of the form $px+qy=rz$ where $p\geq q\geq r$ are fixed positive integers (see~\cite{HT2}).
In particular, our result shows that in the case when $p=q \geq 2$, $r=1$ the lower bound in Proposition~\ref{MSF6}(iii) is correct up to an error term in the exponent.

\section{Concluding remarks}
The results in the paper show that the parameter $f_{\max} (n,\LL)$ can exhibit very different behaviour depending on the linear equation $\LL$.
Indeed, Theorem~\ref{max1} gives a `crude' general upper bound on $f_{\max} (n,\LL)$ for all homogeneous three-variable linear equations $\LL$.
(It is crude in the sense that, in the proof, we do not use any structural information about the link graphs.)
However, this bound is close to the correct value of $f_{\max} (n,\LL)$ for certain equations $\LL$ (Proposition~\ref{MSF2}).
On the other hand, for many equations this bound is far from tight (Theorem~\ref{max2}). Further, for some equations ($x+y=z$ and $2x+2y=z$) the value of $f_{\max} (n,\LL)$ is tied to the property that any triangle-free graph on $n$ vertices contains at most $2^{n/2}$ maximal independent sets. Theorem~\ref{max3} and upper bounds we have obtained since submitting this paper (see~\cite{HT2}) suggest though that the value of $f_{\max} (n,\LL)$ for other equations $\LL$ may depend on completely different factors.
Further progress on understanding the possible behaviour of $f_{\max} (n,\LL)$ would be extremely interesting.


We conclude by briefly describing some results concerning equations with more than three variables. First observe the following simple proposition.

\begin{prop}\label{Mu1} 
Let $\LL_1$ denote the equation $p_1x_1+\dots +p_k x_k =b$ where $p_1, \dots ,p_k ,b \in \mathbb{Z}$ and let $\LL_2$ denote the equation $(p_1+p_2)x_1+p_3 x_2 +\dots +p_k x_{k-1} =b$. Then $\mu_{\LL_1}(n) \leq \mu_{\LL_2}(n)$.
\end{prop}

The proposition is just a simple consequence of the observation that any solution to the equation $\LL_2$ gives rise to a solution to the equation $\LL_1$. So all $\LL_1$-free subsets of $[n]$ are also $\LL_2$-free. Note that for the equations $\LL$ which satisfy the hypothesis of the following corollary, the interval $[\floor{n/(p+q)}+1,n]$ is $\LL$-free. Hence by applying the above proposition along with Corollary~\ref{cormain}, we attain the following result.

\begin{cor}\label{Mu5ai}
Let $\LL$ denote the equation $a_1 x_1 + \cdots + a_k x_k + b_1 y_1 + \cdots + b_{\ell} y_{\ell} = c_1 z_1 + \cdots + c_m z_m$ where the $a_i,b_i,c_i \in \mathbb{N}$ and $p':=\sum_i a_i$, $q':=\sum_i b_i$ and $r':=\sum_i c_i$. Let $t':=\gcd(p',q',r')$ and write $p:=p'/t'$, $q:=q'/t'$ and $r:=r'/t'$. Suppose that $r=1$. Then for sufficiently large $n$, we have $\mu_{\LL}(n)=n-\floor{n/(p+q)}$.
\end{cor}

One can define a link hypergraph $L_S[B]$ analogous to the notion of a link graph defined in Section~\ref{secli} (i.e. now hyperedges correspond to solutions to $\LL$ involving at least one element of $S$). We remark that the removal and container lemmas of Green~\cite{G-R} that we applied do hold for homogeneous linear equations on more than three variables. By arguing as in Lemma~\ref{MSF3} (but by considering a link hypergraph), one can obtain the following simple result.
\begin{prop}\label{MSF8}
Let $\LL$ denote the equation $p_1 x_1 + \dots + p_s x_s = rz$ where $p_1\geq p_2 \geq \dots \geq p_s > r \geq 1$ are positive integers. Then $f_{\max}(n,\LL) \leq f(\floor{rn/p_s},\LL)$.
\end{prop}

In~\cite{HT2} we obtain further results concerning the number of maximal solution-free sets for linear equations with more than three variables. However the proof method does not use structural results such as Theorem~\ref{L5}, and  only work for \emph{some} linear equations. 
Obtaining similar structural results for the number of maximal independent sets in (non-uniform) hypergraphs would help to attain (general) upper bounds for the number of maximal solution-free sets.




\section*{Acknowledgments}
The second author is supported by EPSRC grant EP/M016641/1.
The authors are grateful to Allan Lo for comments that lead to a simplification in the calculations in the proof of Theorem~\ref{structure}.
We are also grateful to Deryk Osthus for comments on the manuscript, and to the referees for careful and helpful reviews.

{\footnotesize \obeylines \parindent=0pt

Robert Hancock, Andrew Treglown
School of Mathematics
University of Birmingham
Edgbaston
Birmingham
B15 2TT
UK
}
\begin{flushleft}
{\emph{E-mail addresses}:
\tt{\{rah410,a.c.treglown\}@bham.ac.uk}}
\end{flushleft}

  \begin{appendix}

\section{}
In this appendix we give the proof of Claim~\ref{claimA} and Proposition~\ref{best}.
\subsection{Proof of Claim~\ref{claimA}}

We use induction on $k$. Recall that $p \geq q \geq 2$. For the base case $k=6$ we directly calculate (\ref{n2}). First note that 

\begin{align*}
& \frac{q^2+1}{q^2+q+1}-\frac{q^7-q^6+q^5-q^4+q^3-q^2+q-1+p(q^6+q^4+q^2+1)} {q^7+p(q^6+q^5+q^4+q^3+q^2+q+1)} \\
= & \frac{(q^6+(p-1)q^5+q^4+(p-1)q^3+q^2+1)}{(q^2+q+1)(q^7+p(q^6+q^5+q^4+q^3+q^2+q+1))}, \\
\end{align*}

and so we have 
\begin{align*}
& \Big(\frac{q^2+1}{q^2+q+1}-C_6\Big)\Big(q^6+p(q^5+q^4+q^3+q^2+q+1) \Big) \\
= & \frac{(q^6+(p-1)q^5+q^4+(p-1)q^3+q^2+1)(q^6+p(q^5+q^4+q^3+q^2+q+1)} {(q^2+q+1)(q^7+p(q^6+q^5+q^4+q^3+q^2+q+1))}. \\
\end{align*}

Since $p \geq q \geq 2$ every power of $q$ in the numerator has a coefficient of at least 1 in both expressions, hence the numerator as a single polynomial in $q$ has positive coefficients. Hence we can make our fraction smaller by dropping lower powers of $q$. We then make further use of $p \geq q \geq 2$ to get the desired result:

\begin{align*}
& \frac{(q^6+(p-1)q^5+q^4+(p-1)q^3+q^2+1)(q^6+p(q^5+q^4+q^3+q^2+q+1)} {(q^2+q+1)(q^7+p(q^6+q^5+q^4+q^3+q^2+q+1))} \\
\geq & \frac{q^{12}+(2p-1)q^{11}+(p^2+1)q^{10}+(p^2+2p-1)q^9}{(q^2+q+1)(q^7+p(q^6+q^5+q^4+q^3+q^2+q+1))} \\
\geq & \frac{q^{12}+(2p-1)q^{11}+(p^2+1)q^{10}+(p^2+2p-1)q^9}{(p+1)q^{10}} \\
= & \frac{q^2+(2p-1)q+(p^2+1)}{p+1}+\frac{p^2+2p-1}{(p+1)q} \geq \frac{p^2+4p+3}{p+1}+\frac{p^2+p}{(p+1)q} = p+3+p/q \geq 6 =k. \\
\end{align*}

For the inductive step, assume that (\ref{n2}) holds for $k$. It suffices to show that $C_k \geq C_{k+1}$ as then the result holds for $k+1$:

\begin{align*}
& \Big(\frac{q^2+1}{q^2+q+1}-C_{k+1}\Big) \Big( q^{k+1}+p\sum\limits_{i=0}^{k} q^i\Big) \geq \Big(\frac{q^2+1}{q^2+q+1}-C_{k}\Big) \Big( q^{k+1}+p\sum\limits_{i=0}^{k} q^i\Big) \\
\geq & \, q \Big(\frac{q^2+1}{q^2+q+1}-C_{k}\Big) \Big( q^{k}+p\sum\limits_{i=0}^{k-1} q^i\Big) \geq qk \geq k+1. \\
\end{align*}

For $k$ even, we have $C_k=C_{k+1}$ by definition. For $k$ odd, consider the following calculations: 

\begin{enumerate}[(i)]
\item{$D_1:=q^{k+2} \Big( \sum\limits_{i=0}^{k} (-1)(-q)^i \Big) - q^k \Big( \sum\limits_{i=0}^{k+2} (-1)(-q)^i \Big)=-q^{k+1}+q^k$,}
\item{$D_2:=pq^{k+2} \Big( \sum\limits_{i=0}^{(k-1)/2} q^{2i} \Big) - pq^k \Big( \sum\limits_{i=0}^{(k+1)/2} q^{2i} \Big)=-pq^k$,}
\item{$D_3:=p \Big( \sum\limits_{i=0}^{k+1} q^i \Big) \Big( \sum\limits_{i=0}^{k} (-1)(-q)^i \Big) - p\Big( \sum\limits_{i=0}^{k-1} q^i \Big) \Big( \sum\limits_{i=0}^{k+2} (-1)(-q)^i \Big)=pq^{k+1}-pq^k$,}
\item{$D_4:=p^2 \Big( \sum\limits_{i=0}^{k+1} q^i \Big) \Big( \sum\limits_{i=0}^{(k-1)/2} q^{2i} \Big) -p^2 \Big( \sum\limits_{i=0}^{k-1} q^i \Big) \Big( \sum\limits_{i=0}^{(k+1)/2} q^{2i} \Big) = p^2 q^k$.}
\end{enumerate}

Using these we have

\begin{align*}
C_k-C_{k+1} = & \frac{ \Big( \sum\limits_{i=0}^{k} (-1)(-q)^i \Big) + p \Big( \sum\limits_{i=0}^{(k-1)/2} q^{2i} \Big)}{q^k+p \Big( \sum\limits_{i=0}^{k-1} q^i \Big)} - \frac{ \Big( \sum\limits_{i=0}^{k+2} (-1)(-q)^i \Big) + p \Big( \sum\limits_{i=0}^{(k+1)/2} q^{2i} \Big)}{q^{k+2}+p \Big( \sum\limits_{i=0}^{k+1} q^i \Big)} \\
= & \frac{D_1+D_2+D_3+D_4}{\Big(q^{k}+p \Big( \sum\limits_{i=0}^{k-1} q^i \Big)\Big)\Big(q^{k+2}+p \Big( \sum\limits_{i=0}^{k+1} q^i \Big)\Big)} \\
= & \frac{(p-1)q^{k+1}+(p^2-2p+1)q^k}{\Big(q^{k}+p \Big( \sum\limits_{i=0}^{k-1} q^i \Big)\Big)\Big(q^{k+2}+p \Big( \sum\limits_{i=0}^{k+1} q^i \Big)\Big)} \geq 0, \\
\end{align*}
where the last inequality follows since $p,q \geq 2$.
\endproof


\subsection{Proof of Proposition~\ref{best}}
Suppose that $\gcd(p,q)=q$. To prove (ii) it suffices to show that 
$$\mu _\LL (n) -\mu ^* _\LL (n) \leq 2 \mu _\LL (\floor{ (n-p)/q}) +o(n).$$
Since $\mu _\LL (n)=(p+q-1)n/(p+q) +o(n)$, $\mu _\LL (\floor{ (n-p)/q})=(p+q-1)n/q(p+q) +o(n)$ and $\mu ^* _\LL (n) =(q-1)^2n/q^2  +o(n)$, it is easy to check that this inequality holds.

To prove (iii) in the case where $t:=\gcd(p,q) \neq q$, it certainly suffices to show that $ 2 \mu _\LL (\floor{ (n-p)/q}) \leq \mu _\LL (n) -\mu ^* _\LL (n)+o(n)$. In this case we have $\mu ^* _\LL (n) =(q-1)(t-1)/(qt)  +o(n)$, and hence it suffices to show that $t \leq (pq+q^2-p-q)/(p+2q-2)$. First note that $t \leq q/2$ and so $q \neq 1$. Now observe that $t(p+2q-2) \leq q(p+2q-2)/2 =pq/2+q^2-q \leq pq+q^2-p-q$ and so our inequality on $t$ holds as required.

To prove (iii) in the case where $\gcd(p,q)=q$ and $p \geq q^2$, it suffices to  show that $$2^{\frac{(p+q-1)n}{(p+q)q}} \leq 3^{\frac{(p+q-1)n}{3(p+q)}-\frac{(q-1)^2 n}{3q^2}}.$$
Let $a:=\log_3(8)$. The inequality can be rearranged to give
$$p((2-a)q-1) \geq (a-1)(q^2-q).$$
If $q \geq 10$ then $((2-a)q-1)$ is positive and so we require $p \geq (a-1)(q^2-q)/((2-a)q-1)$. Note that for $q \geq 18$ this always holds since $p\geq q^2 \geq (a-1)(q^2-q)/((2-a)q-1)$.

To prove (i), suppose that $\gcd(p,q)=q$. It suffices to show that $$ 3^{\frac{(p+q-1)n}{3(p+q)}-\frac{(q-1)^2 n}{3q^2}} \leq 2^{\frac{(p+q-1)n}{(p+q)q}},$$
or rearranging $$p((2-a)q-1) \leq (a-1)(q^2-q).$$
If $q \leq 9$ then $((2-a)q-1)$ is negative and so the inequality holds as the right hand side is non-negative. If $10 \leq q \leq 17$ then the inequality holds if $p \leq (a-1)(q^2-q)/((2-a)q-1)$.
\endproof
\end{appendix}

\end{document}